\pdfoutput=1
\documentclass[10pt]{extarticle}

\usepackage[left=1in,right=1in,top=1in,bottom=1.3in]{geometry}
\usepackage{amsmath, amssymb, amsthm, latexsym, amsfonts, indentfirst, xcolor, mathtools, manfnt, microtype, subcaption}
\usepackage[utf8]{inputenc} 
\usepackage[english]{babel}

\usepackage[breaklinks]{hyperref}
\usepackage{cleveref}
\hypersetup{
	colorlinks = true, 
	urlcolor = cyan, 
	linkcolor = teal, 
	citecolor = cyan 
}

\let\OLDthebibliography\thebibliography
\renewcommand\thebibliography[1]{
	\OLDthebibliography{#1}
	\setlength{\parskip}{0pt}
	\setlength{\itemsep}{0pt plus 0.3ex}
}

\newtheorem{Theorem}{Theorem}

\newtheorem{Lemma}{Lemma}

\newtheorem{Proposition}{Proposition}

\newcommand{\R}{\mathbb R}

\newcommand{\N}{\mathbb N}

\newcommand{\MS}{M\"obius loop}

\begin{document}

\date{}
\title{General penny graphs are at most 43/18-dense}
\author{
	Arsenii Sagdeev\thanks{Alfréd Rényi Institute of Mathematics, Budapest, Hungary. Email:~\href{mailto:sagdeevarsenii@gmail.com}{\tt sagdeevarsenii@gmail.com}.}
}

\maketitle

\begin{abstract}
	We prove that among $n$ points in the plane in general position, the shortest distance occurs at most $43n/18$ times, improving upon the upper bound of $17n/7$ obtained by G.~T\'oth in 1997.
\end{abstract}

\section{Introduction}

One of the first results that modern students learn in graph theory classes is that the maximum number of edges in a planar graph on $n\ge 3$ vertices is $3n-6$. Though the proof is completely elementary, even a small modifications of the problem can bring one to the edge of contemporary graph theory. For example, how does this upper bound change if we consider only \textit{matchstick graphs}, that are planar graphs such that all their edges are straight line segments of the same unit length\footnote{In this paper, we do not distinguish graphs from their drawings for simplicity.}? After a recent partial success~\cite{LS22}, Lavollée and Swanepoel~\cite{LS23} finally solved this problem by settling a forty-year-old conjecture of Harborth~\cite{Har86}.

\begin{Theorem}[\cite{LS23}] \label{tLS}
	The maximum number of edges in a matchstick graph on $n$ vertices is $\lfloor 3n-\sqrt{12n-3} \rfloor$.
\end{Theorem}

A \textit{penny graph} is a matchstick graph such that the distance between every two of their vertices as at least $1$. In other words, for a set of points in the plane, the edges of a penny graph are exactly the shortest distances between the points. For this special class of matchstick graphs, it is much easier to obtain the upper bound from \Cref{tLS} as it was notably shown by Harborth himself in 1974, see~\cite{Har74}. Observe that a hexagonal piece of a regular triangular lattice, which is the extremal configuration achieving this upper bound, contains many collinear triples. Brass wondered~\cite[Section~5.7, Problem~1]{BMP} how this upper bound would change if we considered only penny graphs with vertices in \textit{general position}, that is without collinear triples. In 1997, G.~T\'oth~\cite{Toth97} managed to prove the following.

\begin{Theorem}[G.~T\'oth] \label{tT}
	For each $n \in \N$, there exists a penny graph on $n$ vertices in general position that contains $37n/16-O(\sqrt{n}) = 2.3125n - O(\sqrt{n})$ edges. On the other hand, every  penny graph on $n$ vertices in general position has at most $17n/7 < 2.4286n$ edges.
\end{Theorem}

In this paper, we reduce the gap between these two bounds by improving the upper one. Our main tool here is the \textit{discharging method} that has numerous applications in graph theory \cite{Ack1,Ack2,AT,KP,RT} the most prominent of which is perhaps the proof of the four color theorem~\cite{AH}.

\begin{Theorem} \label{tmain}
	Every penny graph on $n$ vertices in general position has at most $43n/18 < 2.3889n$ edges.
\end{Theorem}

Let us also mention that various closely related problems were extensively studied during the last few decades, see the papers \cite{Be1,Brass1,Epp,GT,Gol,KKKRSU, Mor, PT2, PT,Swan2,Ves} and the books \cite{BMP, HT}.

\vspace{2mm}

\noindent
{\bf Proof outline.} In \Cref{s2}, we study the local structure that can arise between the vertices of a penny graph in general position and their neighbors. After describing some of these local constraints, we apply the discharging method to show that the number of edges cannot exceed $ 12n/5 = 2.4n$, which already improves the aforementioned result of G.~T\'oth. In \Cref{s3}, we extend our analysis to the second neighborhood of each vertex to find new local constraints, for one of which, namely for \Cref{l_no_clover}, we can give only a computer assisted proof. Combined with the discharging method, these new constraints would further improve the upper bound and complete the proof of \Cref{tmain}.

\section{Structure of the first neighborhood} \label{s2}

Most of the structural properties discussed in this section have already been noted in~\cite{Toth97}, though not all of them have been explicitly stated. Nevertheless, we present their (short) proofs here for completeness. We begin with some basic properties that trivially hold for all geometric graphs with edges of the same length.

\begin{Lemma} \label{l_unit} Every triangular face is an equilateral triangle;
		every quadrilateral face is a rhombus.
\end{Lemma}

Our next basic lemma describes two simple properties valid for all penny graphs. Note that throughout this paper, we measure the angle clockwise. In particular, for all vertices $A,B,C$, we have $\angle ABC+\angle CBA = 2\pi$.

\begin{Lemma} \label{l_angles}
	\,
	\begin{itemize}
		\item If $ABC$ is a path, then $\angle ABC \ge \pi/3$. Besides, if $\angle ABC = \pi/3$, then  $AC$ is an edge.
		\item If $ABCD$ is a path, then $\angle ABC + \angle BCD \ge\pi$. Besides, if $\angle ABC + \angle BCD =\pi$, then  $AD$ is an edge.
	\end{itemize}
\end{Lemma}
\begin{proof}
	The first property is immediate from the fact that $|AC| \ge |AB|=|BC|$. To see the second one, observe that if $\angle ABC + \angle BCD = \pi$, then $ABCD$ is a rhombus, and thus $|AD|=|BC|$. However, any rotation of $D$ around $C$ decreasing $\angle BCD$ results in $|AD|<|BC|$ which is not allowed, see \Cref{F1}. 
\end{proof}

\begin{figure}[!htb]
	\begin{minipage}{0.48\textwidth}
		\centering
		\vspace{-2mm}
		\includegraphics[scale=2.0]{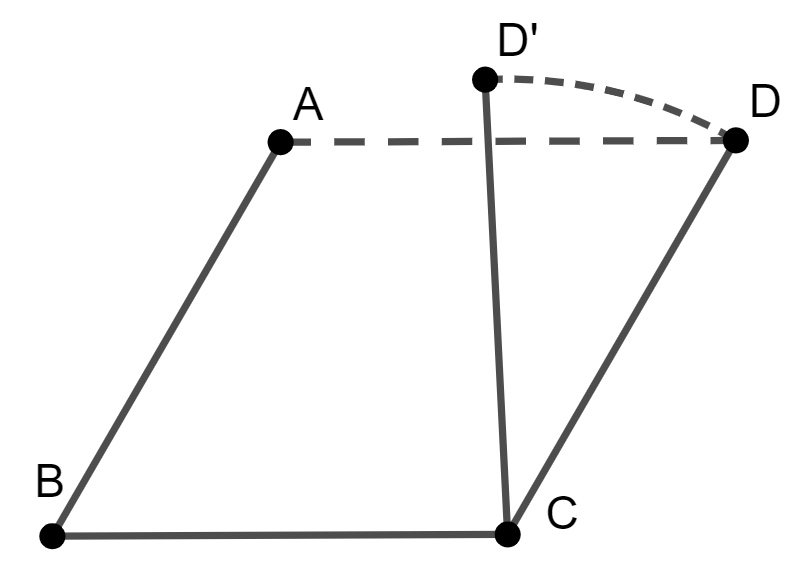}
		\caption{$|AD'| < |BC|$.}
		\label{F1}
	\end{minipage}\hfill
	\begin{minipage}{0.48\textwidth}
		\centering
		\includegraphics[scale=2.0]{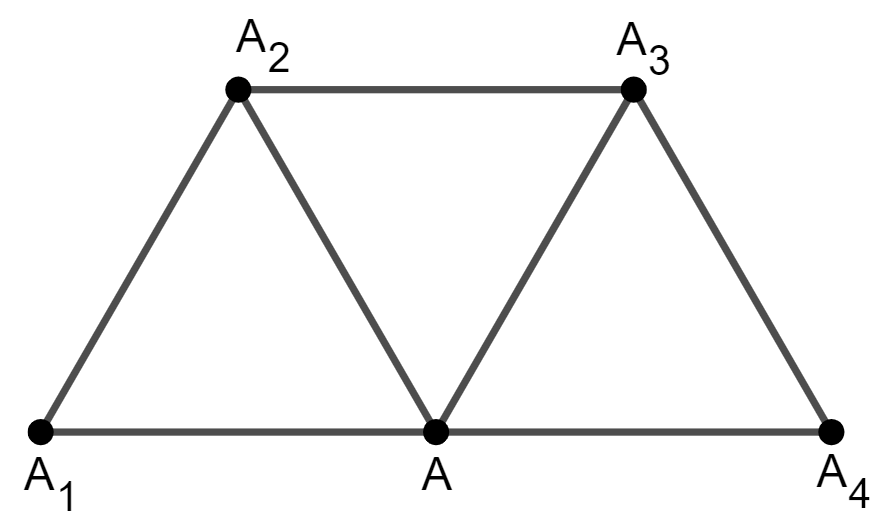}
		\vspace{-1.5mm}
		\caption{$A_1, A$, and $A_4$ are collinear.}
		\label{F2}
	\end{minipage}
\end{figure}

Let us proceed with the properties that are specific for penny graphs with vertices in general position.

\begin{Lemma} \label{l_deg5}
	Among three consecutive faces around each vertex, at least one is not a triangle. Therefore, the degree of each vertex is at most $5$.
\end{Lemma}
\begin{proof}
	Assume the contrary, namely that some vertex $A$ belongs to three consecutive equilateral triangles $AA_1A_2$, $AA_2A_3$, and $AA_3A_4$, see \Cref{F2}. Then the vertices $A_1AA_4$ are collinear, a contradiction. 
\end{proof}

\begin{figure}[!htb]
	\centering
	\vspace{-3mm}
	\includegraphics[scale=2.0]{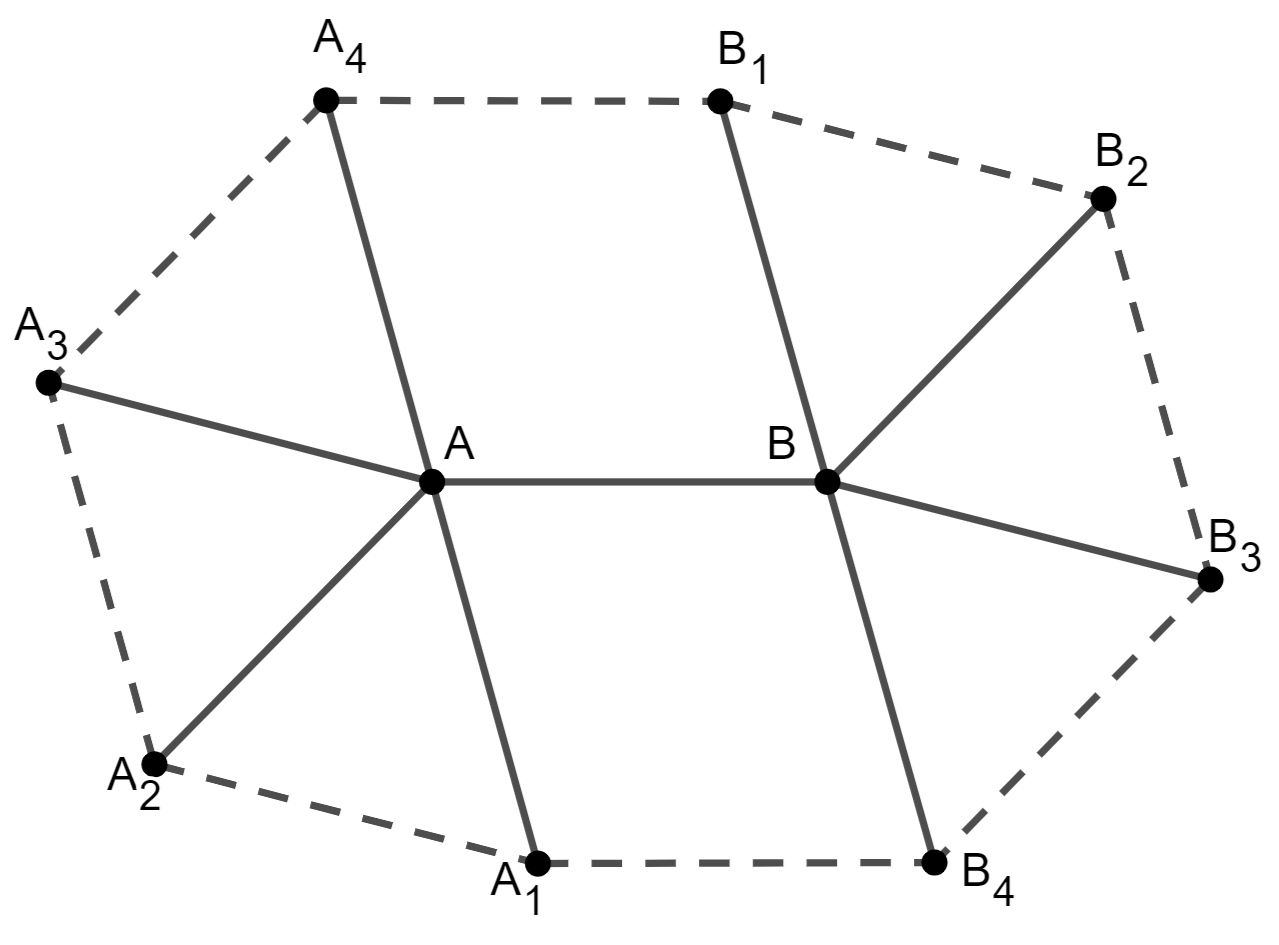}
	\vspace{-2mm}
	\caption{$A_1, A$, and $A_4$ are collinear.}
	\label{F3}
\end{figure}

\begin{Lemma} \label{l_triangle}
	If two adjacent vertices are both of degree $5$, then they have a common neighbor.
\end{Lemma}
\begin{proof}
	Assume the contrary, namely that two adjacent vertices $A$ and $B$ of degree $5$ do not have a common neighbor. We label their remaining neighbors as $A_1, A_2, A_3, A_4$ and $B_1, B_2, B_3, B_4$, respectively, in clockwise order, see \Cref{F3}. Count the sum of all  $10$ angles around $A$ and $B$ in two ways. On the one hand, it is clear that the result is $4\pi$. On the other hand, \Cref{l_angles} implies that $\angle A_4AB + \angle ABB_1 \ge \pi$, $\angle A_1AB + \angle ABB_4 \ge \pi$, while each of the $6$ remaining angles is at least $\pi/3$. These two quantities coincide if and only if all inequalities turn into equalities. However, $A_1, A$, and $A_4$ are collinear in this case, a contradiction.
\end{proof}

\begin{Lemma} \label{l_mobius}
	Let $XYZ$ be a triangle on vertices of degree $5$. Then either some two of them share a common neighbor different from the third vertex, or every two cyclically consecutive of their nine remaining neighbors are adjacent.
\end{Lemma}
\begin{proof}
	Suppose that no two vertices of the triangle $XYZ$ share a common neighbor different from the third vertex. Let $X_1, X_2, X_3, Y_1, Y_2, Y_3, Z_1, Z_2, Z_3$ be all their remaining neighbors, respectively, labeled in clockwise order, see \Cref{F6}. We count the sum of all  $15$ angles around $X,Y$, and $Z$ in two ways. On the one hand, the result is clearly equal to $6\pi$. On the other hand, \Cref{l_angles} implies that each of the $3$ sums $\angle X_3XY + \angle XYY_1, \angle Y_3YZ + \angle YZZ_1$, and $\angle Z_3ZX + \angle ZXX_1$ is at least $\pi$, while each of the $9$ remaining terms is at least $\pi/3$. These two quantities coincide if and only if all inequalities turn into equalities. Therefore, consecutive vertices on the outer cycle of length $9$ are indeed adjacent.
\end{proof}

\vspace{2mm}

\begin{figure}[!htb]
	\begin{minipage}{0.48\textwidth}
		\centering
		\vspace{6mm}
		\includegraphics[scale=1.5]{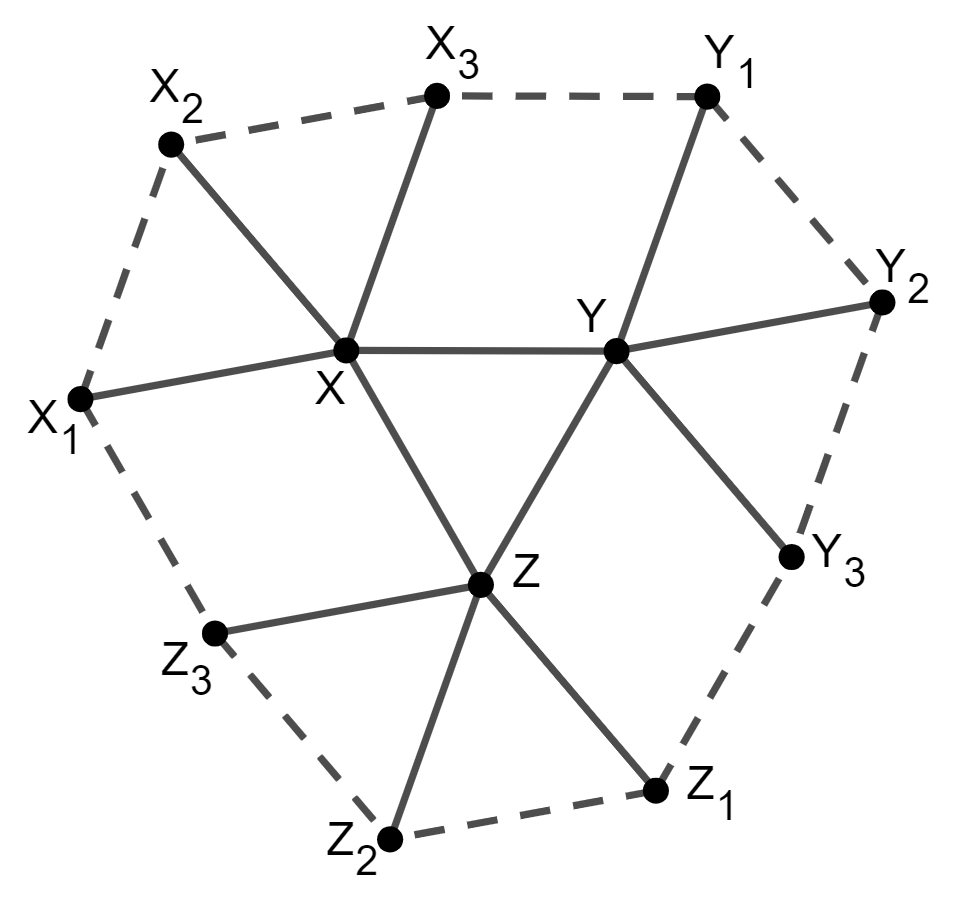}
		\caption{Around a triangle on vertices of degree $5$ without common neighbors, there is a \MS.}
		\label{F6}
	\end{minipage}\hfill
	\begin{minipage}{0.48\textwidth}
		\centering
		\vspace{-11mm}
		\includegraphics[scale=1.5]{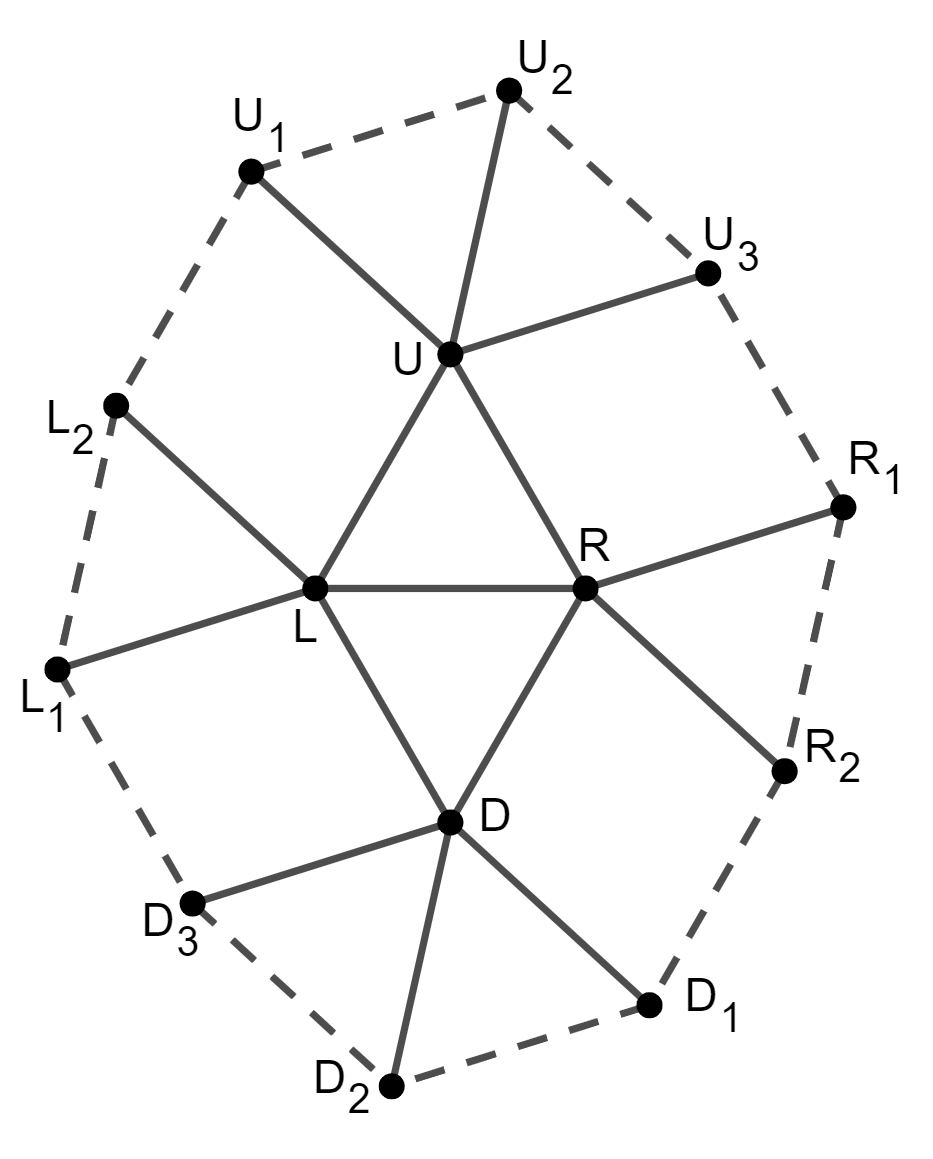}
		\vspace{-1mm}
		\caption{Around each kernel, there is an apricot.}
		\label{F4}
	\end{minipage}
\end{figure}

\vspace{-1mm}
In what follows, we call a configuration in the right side of \Cref{F6} a \textit{\MS}. Our next lemma is very similar to the previous one, since it also ensures some rigid structure around a group of four vertices, each of degree $5$, forming a $4$-cycle with one diagonal.

\begin{Lemma} \label{l_apricot}
	Let $L, U, R, D$ be four vertices of degree $5$ such that five pairs $LU, UR, RD, DL, LR$ are edges. Then every two cyclically consecutive of their ten remaining neighbors are adjacent.
\end{Lemma}
\begin{proof}
	First, observe that the vertices $L, U, R, D$ do not have other common neighbors, since otherwise there would be three consecutive triangular faces around one of these vertices, which contradicts \Cref{l_deg5}. Let $L_1, L_2, U_1, U_2, U_3, R_1, R_2, D_1, D_2, D_3$ be all their remaining neighbors, respectively, labeled in clockwise order, see \Cref{F4}. We count the sum of all  $20$ angles around the four given vertices in two ways. On the one hand, the result is clearly equal to $8\pi$. On the other hand, \Cref{l_angles} implies that each of the $4$ sums $\angle L_2LU + \angle LUU_1, \angle U_3UR + \angle URR_1, \angle R_2RD + \angle RDD_1, \angle D_3DL + \angle DLL_1$ is at least $\pi$, while each of the $12$ remaining terms is at least $\pi/3$. These two quantities coincide if and only if all inequalities turn into qualities. Therefore, consecutive vertices on the outer cycle of length $10$ are indeed adjacent.
\end{proof}

In what follows, we call a quadruple $L, U, R, D$ satisfying the condition of \Cref{l_apricot} a \textit{kernel}, and the cycle of length $10$ around them an \textit{apricot}.

\begin{Lemma} \label{l_no_triangle}
	For each apricot labeled as in the proof of \Cref{l_apricot}, the vertices $L_2$ and $U_1$ do not share a common neighbor. In particular, at least one of them is of degree less than $5$. The same holds for three other symmetric pairs $U_3$ and $R_1$, $R_2$ and $D_1$, $D_3$ and $L_1$ as well.
\end{Lemma}
\begin{proof}
	Without loss of generality, it is sufficient to prove the statement only for one out of four symmetric pairs. Assume the contrary, namely that the vertices $L_2$ and $U_1$ have a common neighbor $A$. Observe that $\angle AL_2L = \angle DLL_2$, see \Cref{F5}. Hence, the segment $AL_2$ is parallel to $LD$. Similarly, $L_2L$ is parallel to $DD_1$ because $\angle DLL_2 = \angle LDD_1$. Since these four segments are also of the same length, we conclude that the vertices $A$, $L$, and $D_1$ are collinear, a contradiction. Now \Cref{l_triangle} implies that either $L_2$ or $U_1$ is of degree less than $5$.
\end{proof}

\begin{figure}[!htb]
	\centering
	\includegraphics[scale=1.50]{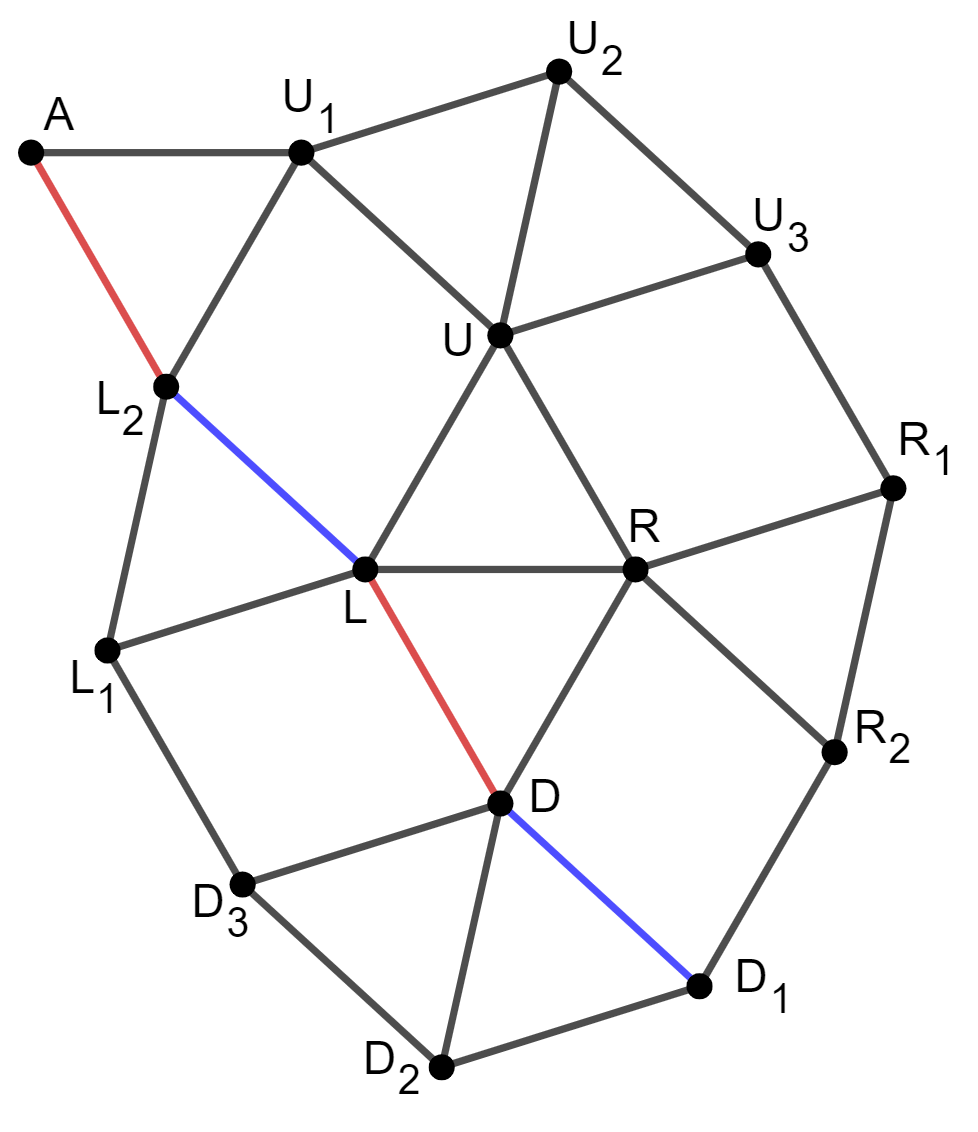}
	\caption{$A, L$, and $D_1$ are collinear.}
	\label{F5}
\end{figure}

\begin{Lemma} \label{l_kernel_dis}
	For each apricot labeled as in \Cref{l_apricot}, the vertices $L_1, L_2, U_1, U_3, R_1, R_2, D_1, D_3$ do not belong to a kernel of another apricot\footnote{Note that $U_2$ and $D_2$ can belong to a kernel. For instance, $U_2$ may play the role of $D$ in another apricot.}. In particular, kernels of distinct apricots are disjoint.
\end{Lemma}
\begin{proof}
	Note that each vertex of a kernel has the following property: among any two consecutive adjacent faces, at least one is a triangle, see \Cref{F4}. On the other hand, by \Cref{l_no_triangle}, none of the eight vertices in the statement have this property.
\end{proof}

\begin{Lemma} \label{l_5_in_core}
	A vertex of degree $5$ such that all its neighbors are also of degree $5$ belongs to a kernel.
\end{Lemma}
\begin{proof}
	Let $A_1, A_2, A_3, A_4, A_5$ be the neighbors $A$, each of degree $5$. For all $i$, $A$ and $A_i$ are adjacent vertices of degree $5$. Therefore, \Cref{l_triangle} implies that they have a common neighbor, i.e. that $A_i$ is adjacent to $A_j$ for some $j\neq i$. Since five neighbors of $A$ cannot be partitioned into disjoint pairs, some must form a triple, that is $A_i$ is adjacent to $A_j$, while $A_j$ is adjacent to $A_k$ for some $i,j,k$. It remains only to note that a quadruple $A, A_i, A_j, A_k$ is a kernel, as desired.
\end{proof}

To conclude this section, we show that the structural properties mentioned above are already enough to improve the upper bound from \Cref{tT}.

\begin{Theorem} \label{tmain_weak}
	Every penny graph on $n$ vertices in general position with $e$ edges satisfies $e/n \le 12/5$.
\end{Theorem}
\begin{proof}
	We argue by the discharging method. Initially, we place $ch(A) \coloneqq 5-\deg(A)$ units of charge onto each vertex $A$. On the first discharging stage, each vertex $A$ of degree at most $4$ gives $q$ units of charge to each of its neighbors of degree $5$, where the value of $q$ will be specified later. On the second discharging stage, we redistribute the charges within each kernel evenly. To give a lower bound on the resulting charge $ch'(A)$ of a vertex $A$ at the end of this procedure, let us consider the following three cases separately.
	\begin{itemize}
		\item If $\deg(A)\le 4$, then it is easy to see that $ch'(A) \ge ch(A) - \deg(A)q\ge 1-4q$.
		\item If a vertex of degree $5$ does not belong to any kernel, then \Cref{l_5_in_core} guarantees that it gets at least $q$ units of charge at the first phase. Moreover, the second phase does not affect its charge at all.
		\item \Cref{l_no_triangle} implies that each kernel receive at least $4q$ units of charge at the first stage. Since different kernels are disjoint by \Cref{l_kernel_dis}, we conclude that after the even redistribution at the second phase, each vertex of a kernel has at least $q$ units of charge.
	\end{itemize}

	\noindent
	To balance these lower bounds, we pick the value of $q$ such that $1-4q=q$, namely $q=1/5$. This choice yields that $ch'(A) \ge 1/5$ for every vertex $A$. Now the degree sum formula implies that
	\begin{equation*}
		5n-2e = \sum_{A} \big(5-\deg(A)\big) = \sum_{A} ch(A) = \sum_{A} ch'(A) \ge \sum_{A} \frac{1}{5} = \frac{n}{5},
	\end{equation*}
	which is equivalent to the desired inequality $e/n \le 12/5$.
\end{proof}

\section{Structure of the second neighborhood: proof of Theorem~\ref{tmain}} \label{s3}

\vspace{-1.5mm}
Let us call a vertex of degree $5$ \textit{popular} if at most one of its neighbors is of degree $4$, while all the other neighbors are of degree $5$, and let us call a vertex \textit{unpopular} otherwise. In other words, we will call a vertex $A$ \textit{unpopular} if either
\begin{itemize}
	\item $\deg(A) \le 4$, or
	\item $\deg(A)=5$ and at least one of its neighbors is of degree at most $3$, or
	\item $\deg(A)=5$ and at least two of its neighbors are of degree $4$.
\end{itemize} 
The following theorem is our main result about the structure of the second neighborhood.

\begin{Theorem} \label{tn2}
	Every vertex of degree $4$ in a penny graph with vertices in general position has an unpopular neighbor that does not belong to a kernel.
\end{Theorem}

Prior to giving the proof, let us show that this result, combined with the discharging argument from the previous section, indeed yields \Cref{tmain}. As earlier, we place $ch(A) \coloneqq 5-\deg(A)$ units of charge onto each vertex $A$. The first discharging stage is a bit different this time: each vertex $A$ of degree at most $4$ gives $q$ units of charge to each neighbor $B$ of degree $5$ if and only if either $B$ belongs to a kernel, or $B$ does not have another neighbor of degree at most $4$ distinct from $A$; otherwise $A$ gives $q/2$ units of charge to $B$. On the second discharging stage, we again redistribute the charges within each kernel evenly. Let us now analyze the resulting charge $ch'(A)$ of a vertex $A$ at the end of this procedure.
\begin{itemize}
	\item If $\deg(A)\le 4$, then \Cref{tn2} guarantees that $A$ cannot give more than $3q+q/2$ units of charge. Hence, $ch'(A) \ge ch(A) - 7q/2\ge 1-7q/2$.
	\item If a vertex of degree $5$ does not belong to a kernel, then \Cref{l_5_in_core} ensures that it has at least one neighbor of degree at most $4$. If such a neighbor is unique, then it gives $q$ units of charge to the vertex, while if there are at least $2$ such neighbors, then each of them gives $q/2$ units of charge. In any case, the vertex gets at least $q$ units of charge at the first phase, and does not lose anything afterwards.
	\item As in the previous section, \Cref{l_no_triangle} implies that each kernel receive at least $4q$ units of charge at the first stage. Since different kernels are disjoint by \Cref{l_kernel_dis}, we conclude that after the even redistribution at the second phase, each vertex of a kernel has at least $q$ units of charge.
\end{itemize}

\noindent
To balance these lower bounds, we pick the value of $q$ such that $1-7q/2=q$, namely $q=2/9$. This choice yields that $ch'(A) \ge 2/9$ for every vertex $A$. Finally, the degree sum formula implies that
\begin{equation*}
	5n-2e = \sum_{A} \big(5-\deg(A)\big) = \sum_{A} ch(A) = \sum_{A} ch'(A) \ge \sum_{A} \frac{2}{9} = \frac{2}{9}n,
\end{equation*}
i.e. that $e/n \le 43/18$, and completes the proof of \Cref{tmain}. Thus it remains only to verify the structural property from \Cref{tn2}, to which we devote the rest of this paper.

\subsection{Proof of Theorem~\ref{tn2}}

Throughout this section, let $A$ be a vertex of degree $4$ in a penny graph with vertices in general position. Our proof that $A$ has an unpopular neighbor that does not belong to a kernel consists of several special cases the last of which is even computer-assisted. The first and the simplest case is when $A$ belongs to an apricot.

\begin{Proposition} \label{p_near_kernel}
	\hspace{-1mm}If $A$ belongs to an apricot, then $A$ has an unpopular neighbor that does not belong to a kernel.
\end{Proposition}
\begin{proof}
	Label the apricot as in \Cref{l_apricot}. Due to the symmetry of this construction, we can assume without loss of generality that $A$ plays the role of either $L_2$, or $U_1$, or $U_2$, see \Cref{F4}.
	
	In the latter case, note that by \Cref{l_no_triangle}, either $L_2$ or $U_1$ is of degree less than $5$. That is, either $U_1$ is of degree at most $4$ by itself, or it has at least two neighbors of degree at most $4$: $L_2$ and $U_2=A$. In either event, $U_1$ is unpopular. Besides that, $U_1$ does not belong to a kernel by \Cref{l_kernel_dis}, as desired.
	
	In case $A=L_2$, we apply the same argument to the pair $D_3$ and $L_1$, which yields that $L_1$ is unpopular and does not belong to a kernel.
	
	The last remaining case $A=U_1$ is trickier. We show that $L_2$ is always the desired unpopular vertex that does not belong to a kernel. As earlier, if $L_1$ or $L_2$ is of degree less than $5$, then we are done. So we can assume without loss of generality that $\deg(L_1)=\deg(L_2)=5$. Apply \Cref{l_mobius} to the triangle $LL_1L_2$. If there is a \MS{} around it, then $L_2$ and $U_1$ share a common neighbor, which contradicts \Cref{l_no_triangle}. So \Cref{l_mobius} implies that $L_1$ and $L_2$ share a common neighbor different from $L$ that we label as $B$, see \Cref{F7}. If $\deg(B)=5$, then $L_2,L,L_1,B$ is a kernel sharing a vertex $L$ with another kernel $L,U,R,D$, which contradicts \Cref{l_kernel_dis}, and thus\footnote{Note that $B$ was called a `special second neighbor of $L$' in \cite{Toth97}, where it was also observed that if $\deg(L_1)=\deg(L_2)=5$, then $\deg(B) \le 4$.} $\deg(B) \le 4$. Hence, $L_2$ is unpopular, because it has at least two neighbors of degree at most $4$: $B$ and $U_1=A$.
\end{proof}

\begin{figure}[!htb]
	\centering
	\includegraphics[scale=1.50]{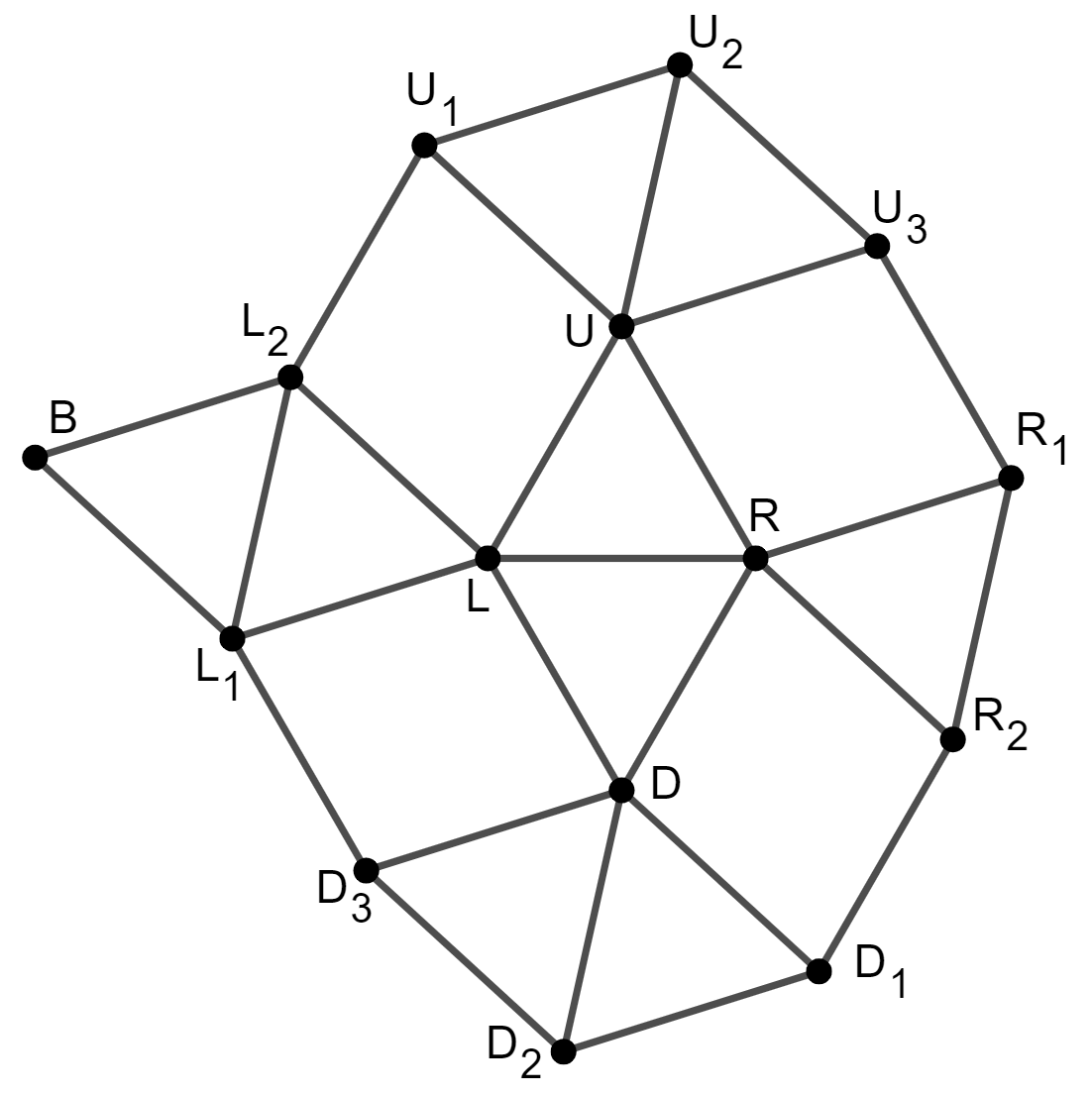}
	\caption{If $\deg(L_1)=\deg(L_2)=\deg(B)=5$, then $L_2,L,L_1,B$ is a kernel.}
	\label{F7}
\end{figure}

So in the rest of this proof, we can assume without loss of generality that $A$ does not belong to an apricot, and thus its neighbors do not belong to a kernel. Then our aim is just to show that $A$ has at least one unpopular neighbor. To this end, let us study the local structure between $A$ and its popular neighbors.

Suppose that $B$ is a popular vertex, and label its neighbors, including $A$, as $C_1, C_2, C_3, C_4, C_5$ in clockwise order. Note that if some $C_i$ is not adjacent to any other neighbor of $B$, then \Cref{l_triangle} implies that $\deg(C_i) \le 4$. Since $B$ is popular, it has at most one such `lonely' neighbor.

If exactly one $C_i$ is `lonely', then it must be $A$, while four remaining neighbors of $B$ split into two pairs of adjacent vertices, see \Cref{F81}. In this case, we say that the edge $AB$ is of Type~I.

If every $C_i$ is adjacent to some other neighbor of $B$, then after a possible cyclic renumbering, we can assume without loss of generality that $C_1$ is adjacent to $C_2$, while both $C_3$ and $C_5$ are adjacent to $C_4$. If all three vertices $C_3$, $C_4$ and $C_5$ are of degree $5$, then $B,C_3,C_4,C_5$ is a kernel, which contradicts our assumption that $A$ does not belong to an apricot. Therefore, there are only two different cases up to the symmetry: $A=C_5$, in which we say that the edge $AB$ is of Type~II, and $A=C_4$, in which we say that $AB$ is of Type~III, see Figures~\ref{F82} and~\ref{F83}, respectively.

\begin{figure}[!htb]
	\centering
	\begin{subfigure}[b]{.32\linewidth}
		\centering
		\includegraphics[scale=1.60]{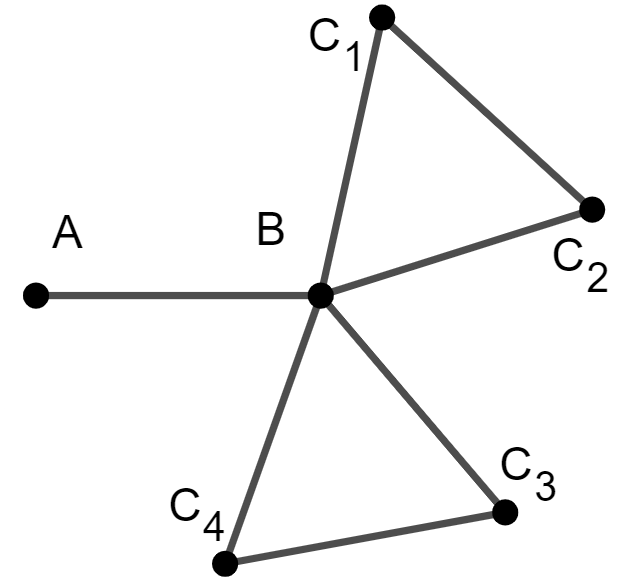}
		\caption{Type I}
		\label{F81}
	\end{subfigure}
	\begin{subfigure}[b]{.32\linewidth}
		\centering
		\includegraphics[scale=1.60]{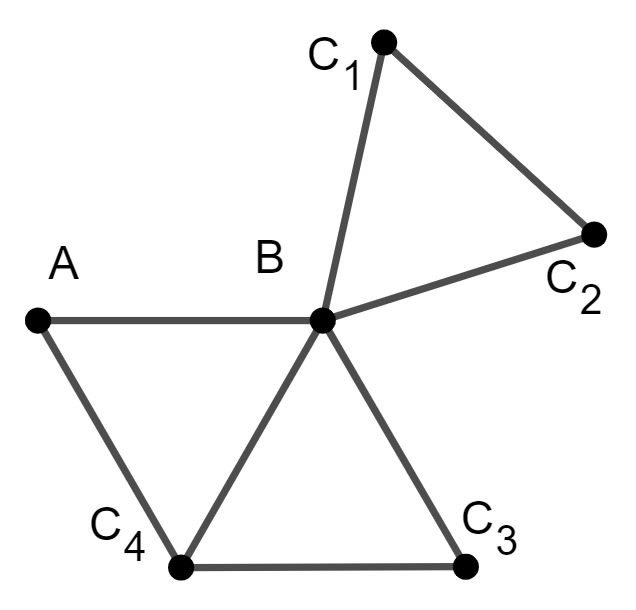}
		\caption{Type II}
		\label{F82}
	\end{subfigure}
	\begin{subfigure}[b]{.32\linewidth}
		\centering
		\includegraphics[scale=1.60]{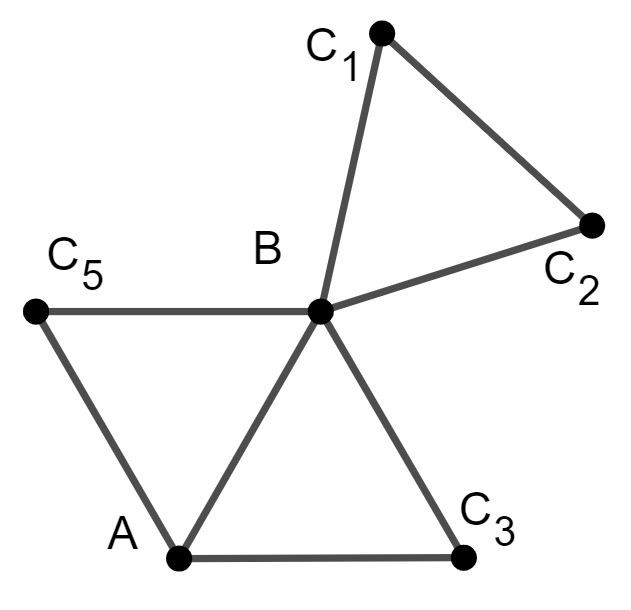}
		\caption{Type III}
		\label{F83}
	\end{subfigure}
	\captionsetup{justification=centering}
	\caption{All three possible edge types between a vertex $A$ of degree $4$ that does not belong to an apricot \newline and its popular neighbor $B$.}
	\label{F8}
\end{figure}

It is almost trivial to find an unpopular neighbor of $A$ (even two of them) when at least one edge coming from $A$ is of Type~III.

\begin{Proposition} \label{p_typeiii}
	In the notation of \Cref{F83}, both $C_3$ and $C_5$ are unpopular.
\end{Proposition}
\begin{proof}
	The edge $AC_3$ cannot be of Type~I, because $A$ and $B$ are not disjoint. If this edge is of Type~II, then $B$ and $C_3$ share a common neighbor different from $A$. If $AC_3$ is of Type~III, then $A$ and $C_3$ share a common neighbor different from $B$. Both these alternatives lead to three consecutive triangles, which contradicts \Cref{l_deg5}. Since all three possible types are excluded, we deduce that $C_3$ must be unpopular. The same argument works for $C_5$ as well.
\end{proof}

The other two types of edges require much more work to handle. Note that if the edge $AB$ is of Type~II, then $A$ and $B$ share a common neighbor, the vertex $C_4$ in \Cref{F82}. Moreover, $A$ and $C_4$ also share a common neighbor $B$, and thus the edge $AC_4$ is not of Type~I. Therefore, if each of the four edges of $A$ is either of Type~I or of Type~II, then the latter ones go in pairs, i.e. there are only three possible cases: all four edges are of Type~II, all four edges are of Type~I, or two cyclically consecutive edges of each type. In both the second and the third cases, there are two cyclically consecutive edges of Type~I. We begin by arguing that this is not possible.  

\begin{figure}[!htb]
	\centering
	\begin{subfigure}[b]{.45\linewidth}
		\centering
		\includegraphics[scale=1.70]{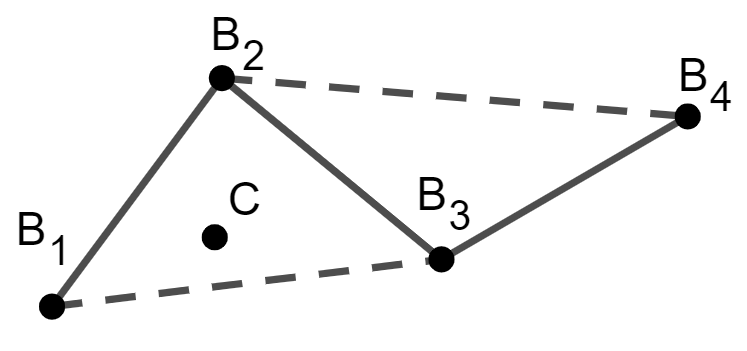}
		\caption{Non-convex path}
		\label{F91}
	\end{subfigure}
	\begin{subfigure}[b]{.45\linewidth}
		\centering
		\includegraphics[scale=1.70]{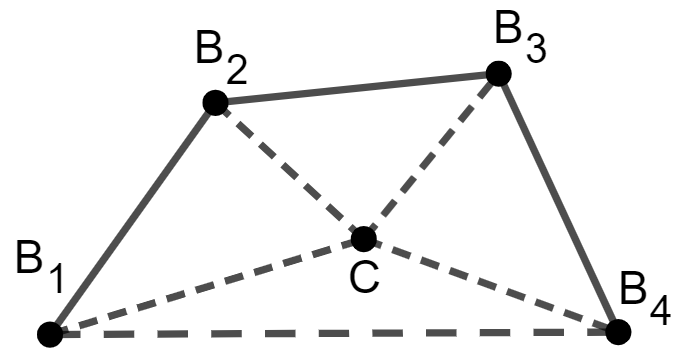}
		\caption{Convex path}
		\label{F92}
	\end{subfigure}
	\captionsetup{justification=centering}
	\caption{A vertex $C$ inside the convex hull of the path $B_1B_2B_3B_4$.}
	\label{F9}
\end{figure}

\begin{Lemma} \label{l_convex}
	The convex hull of a path $B_1B_2B_3B_4$ does not contain another vertex.
\end{Lemma}
\begin{proof}
	Assume the contrary, i.e. that some vertex $C$ belongs to the convex hull. Our goal is to come to a contradiction by showing that the distance $|CB_i|$ is `too short' for some $i$.
	
	If the clockwise order of the vertices in the convex hull is different from their order in the path, as in \Cref{F91}, then we can assume without loss of generality that $C$ falls into the triangle $B_1B_2B_3$. In this case, it is easy to see that $|CB_2| < |B_1B_2|=|B_3B_2|$ since the unit circle centered at $B_2$ is convex, as desired.
	
	Otherwise, note that $\angle B_1CB_2 + \angle B_2CB_3 + \angle B_3CB_4 > \pi$, see \Cref{F92}, and thus at least one of the terms is greater than $\pi/3$. (Recall that these vertices of a penny graph are in general position, and thus the inequality must indeed be strict). If $\angle B_1CB_2 > \pi/3$, then this angle is not the smallest in the triangle $B_1CB_2$ and $B_1B_2$ is not its shortest side, as desired. The other two cases are symmetric.
\end{proof}

\begin{Lemma} \label{l_no_kifli}
	A penny graph with vertices in general position cannot contain the configuration depicted in \Cref{F101} as a subgraph.  
\end{Lemma}

\begin{figure}[!htb]
	\centering
	\begin{subfigure}[b]{.45\linewidth}
		\centering
		\includegraphics[scale=1.50]{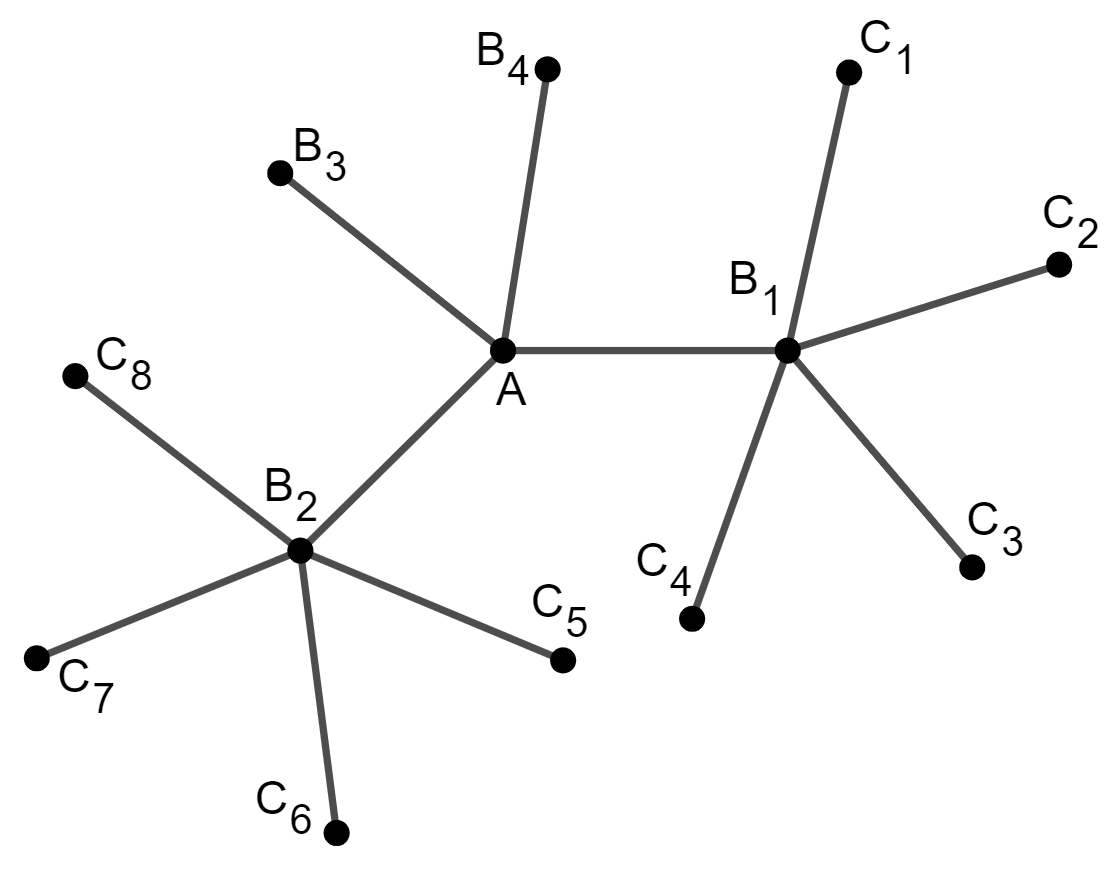}
		\caption{Before rotations}
		\label{F101}
	\end{subfigure}
	\begin{subfigure}[b]{.45\linewidth}
		\centering
		\includegraphics[scale=1.50]{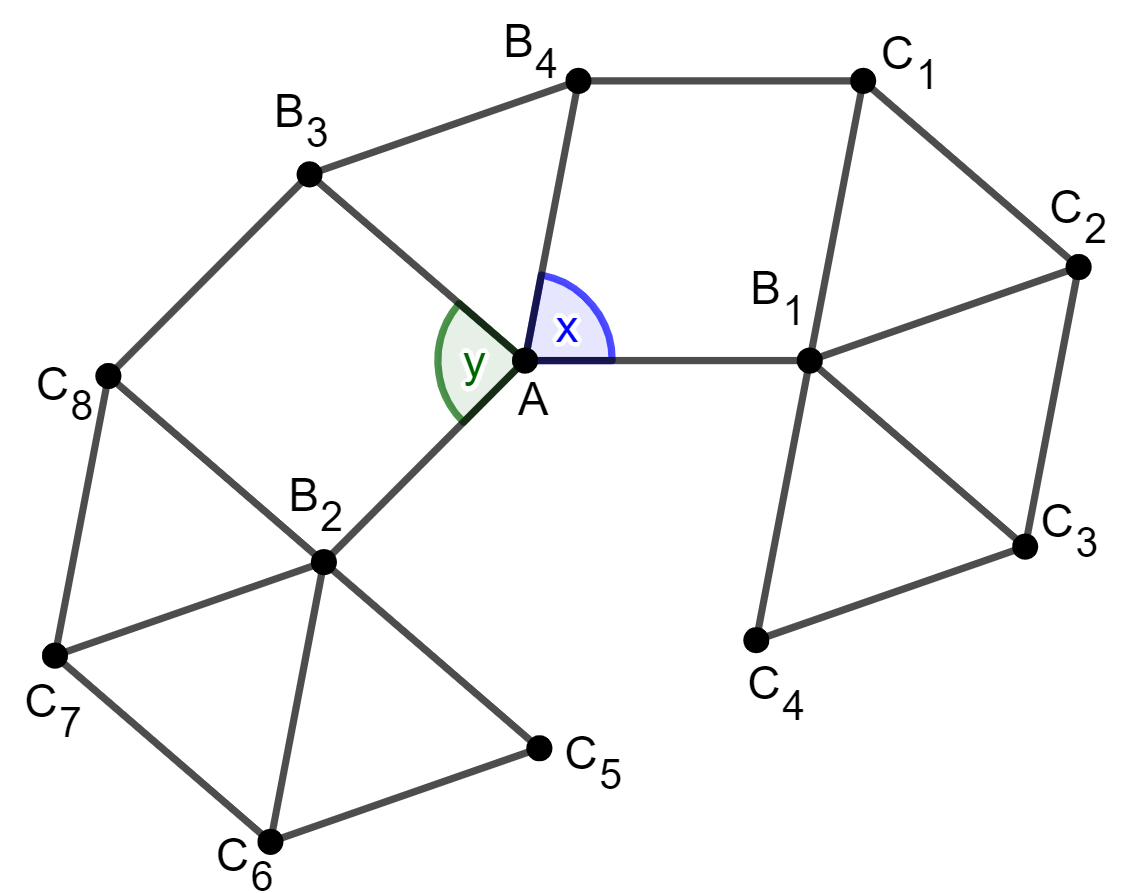}
		\caption{After rotations}
		\label{F102}
	\end{subfigure}
	\captionsetup{justification=centering}
	\caption{Forbidden 13-point configuration.}
	\label{F10}
\end{figure}

\begin{proof}
	Assume the contrary, namely that it is possible to place 13 points in general position on the plane such that all the edges in \Cref{F101} are of unit length, while all the other pairwise distances are not shorter. Let us rotate some parts of this configuration to reduce the number of `degrees of freedom' and simplify the further analysis.
	
	First, we rotate $C_1$ around $B_1$ towards $B_4$ until a new edge appears. If this edge is $C_1B_4$, then $AB_1C_1B_4$ is a rhombus, and we proceed to the next step. If this edge is $AC_1$, then we rotate $B_4$ around $A$ towards $C_1$ until $AB_1C_1B_4$ is a rhombus.
	
	Second, for each $i=2,3,4,$ we consecutively rotate $C_i$ around $B_1$ until the edge $C_iC_{i-1}$ appears. Note that the path $C_5B_2AB_1C_4$ is \textit{convex}\footnote{That it, this five vertices are in convex position in this particular order.} by \Cref{l_convex}. Therefore, during the whole rotation procedure, $C_3$ and $C_5$ lie on the opposite sides of the line $B_1C_4$, and thus the distance $|C_4C_5|$ increases.
	
	Third, we repeat the previous two steps for the other half of our construction, that is to the points $B_2, B_3, C_8, C_7, C_6, C_5$ instead of $B_1, B_4, C_1, C_2, C_3, C_4$, respectively.
	
	Finally, we rotate 6 points $B_2, B_3, C_8, C_7, C_6, C_5$ around $A$ towards $B_4$ with the same angular speed until $B_3B_4$ is also an edge. As earlier, the path $C_5B_2AB_1C_4$ is convex, and thus during the whole rotation procedure, $B_2$ and $C_4$ lie on the opposite sides of the line $AC_5$, so the distance $|C_4C_5|$ increases.
	
	This series of rotations transforms our initial configuration in \Cref{F101} to a much more structured one depicted in \Cref{F102} increasing the distance $|C_4C_5|$ on the way\footnote{An anonymous referee noted that the increase of $|C_4C_5|$ is also immediate from Cauchy’s Arm Lemma, see~\cite[page~228]{Crom}.}. If we denote $\angle B_4AB_1$ and $\angle B_2AB_3$ by $x$ and $y$, respectively, then via a straightforward yet tedious calculation, one can verify that
	\begin{equation*}
		|C_4C_5|^2 = 3-2\sin\Big(x+\frac{\pi}{6}\Big)-2\sin\Big(y+\frac{\pi}{6}\Big)-2\cos\Big(x+y+\frac{\pi}{3}\Big).
	\end{equation*}
	Note that $\pi/3 \le x,y \le 2\pi/3$, since in each of the two corresponding rhombuses, the shorter diagonal cannot be sorter than the side. If either $x=\pi/3$ or $y=\pi/3$, then it is easy to see that $|C_4C_5|^2=1$. Otherwise, if at least one of the two variables, say $x$, is strictly larger than $\pi/3$, then
	\begin{equation*}
		\frac{\partial}{\partial y} |C_4C_5|^2 = 2\sin\Big(x+y+\frac{\pi}{3}\Big)-2\cos\Big(y+\frac{\pi}{6}\Big) = -4\sin\Big(\frac{x}{2}-\frac{\pi}{6}\Big)\sin\Big(y+\frac{x}{2}\Big),
	\end{equation*}
	which is strictly negative on the interval $\pi/3<y<2\pi/3$.
	
	As a result, we conclude that before the rotation procedure, the distance $|C_4C_5|$ could be no less than $1$ if and only if our initial configuration coincides with the one in \Cref{F102} and at least one of the two angles $x$ and $y$ equals $\pi/3$. However, this configuration is definitely not in general position, since, e.g., three points $C_1$, $B_1$, and $C_4$ are collinear, a contradiction.
\end{proof}

\begin{Proposition} \label{p_typei}
	Two cyclically consecutive edges of $A$ cannot both be of Type~I.
\end{Proposition}
\begin{proof} 
	Assume the contrary, namely that among four neighbors of $A$, which we label $B_1,B_2,B_3,$ and $B_4$ in clockwise order, the first two are popular, and both edges $AB_1$ and $AB_2$ are of Type~I. Note that $B_1$ and $B_2$ have another common neighbor different from $A$, since otherwise our penny graph contains the configuration depicted in \Cref{F101} as a subgraph, which contradicts \Cref{l_no_kifli}.
	
	Let us label the neighbors of $B_1$ and $B_2$ as in \Cref{F11}. Recall that $B_1$ is popular, and so we can apply \Cref{l_mobius} to the triangle $B_1C_3C_4$. Since vertices in each of the four pairs $A$ and $C_1$, $C_2$ and $C_3$, $A$ and $C_4$, $B_1$ and $B_2$ are not adjacent, \Cref{l_mobius} implies that $C_4$ and $C_3$ have a common neighbor different from $B_1$, which we denote by $D_1$. Similarly, $C_4$ and $C_5$ must have a common neighbor $D_2$ different from $B_2$. It remains only to observe that whether $D_1=D_2$ or not, we get a contradiction with \Cref{l_deg5}.
\end{proof}

\begin{figure}[!htb]
	\centering
	\includegraphics[scale=1.40]{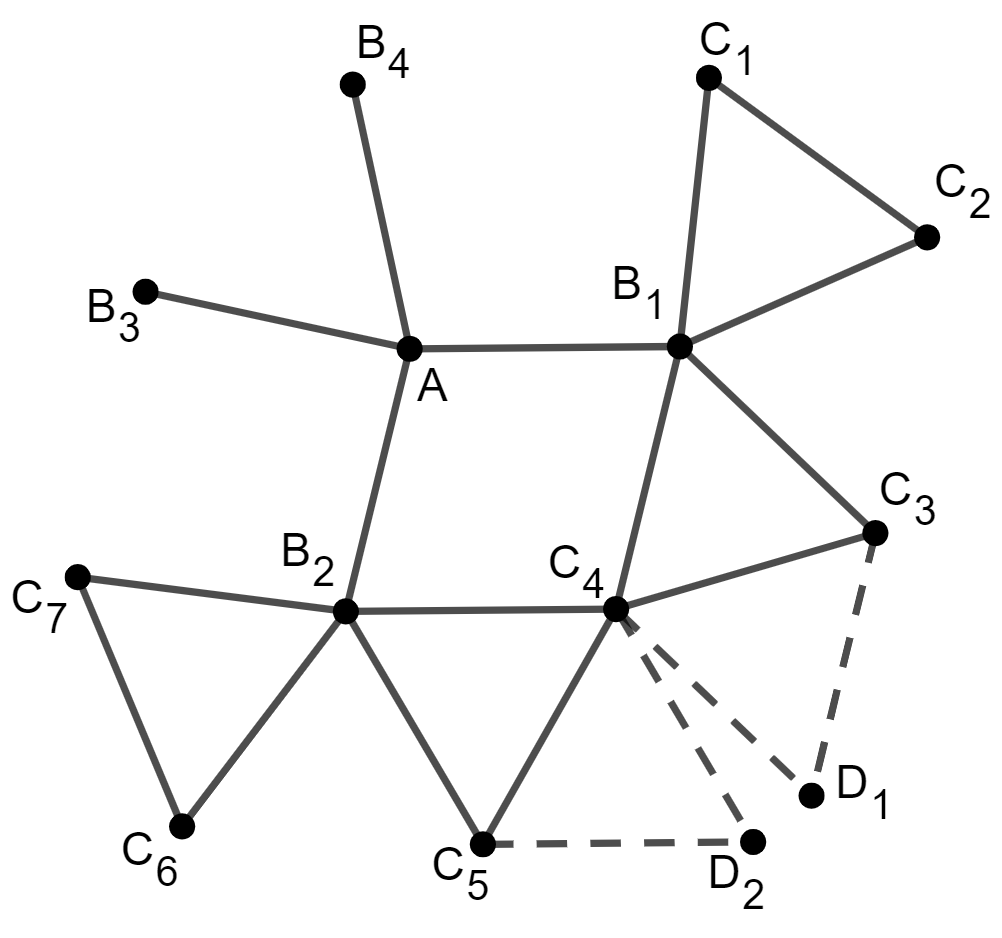}
	\caption{Two consecutive edges of Type~I around a vertex $A$ of degree $4$ is a forbidden configuration.}
	\label{F11}
\end{figure}

Note that the only possibility for $A$ to have four popular neighbors that we have not excluded so far is when all four edges of $A$ are of Type~II. To complete the proof of \Cref{tn2} by excluding this last possibility too, we need the following structural statement. Though it shares many similarities with \Cref{l_no_kifli}, the proof is more technical and computer assisted this time.

\begin{Lemma} \label{l_no_clover}
	A penny graph with vertices in general position cannot contain the configuration depicted in \Cref{F121} as a subgraph.  
\end{Lemma}

\begin{figure}[!htb]
	\centering
	\begin{subfigure}[b]{.45\linewidth}
		\centering
		\includegraphics[scale=1.40]{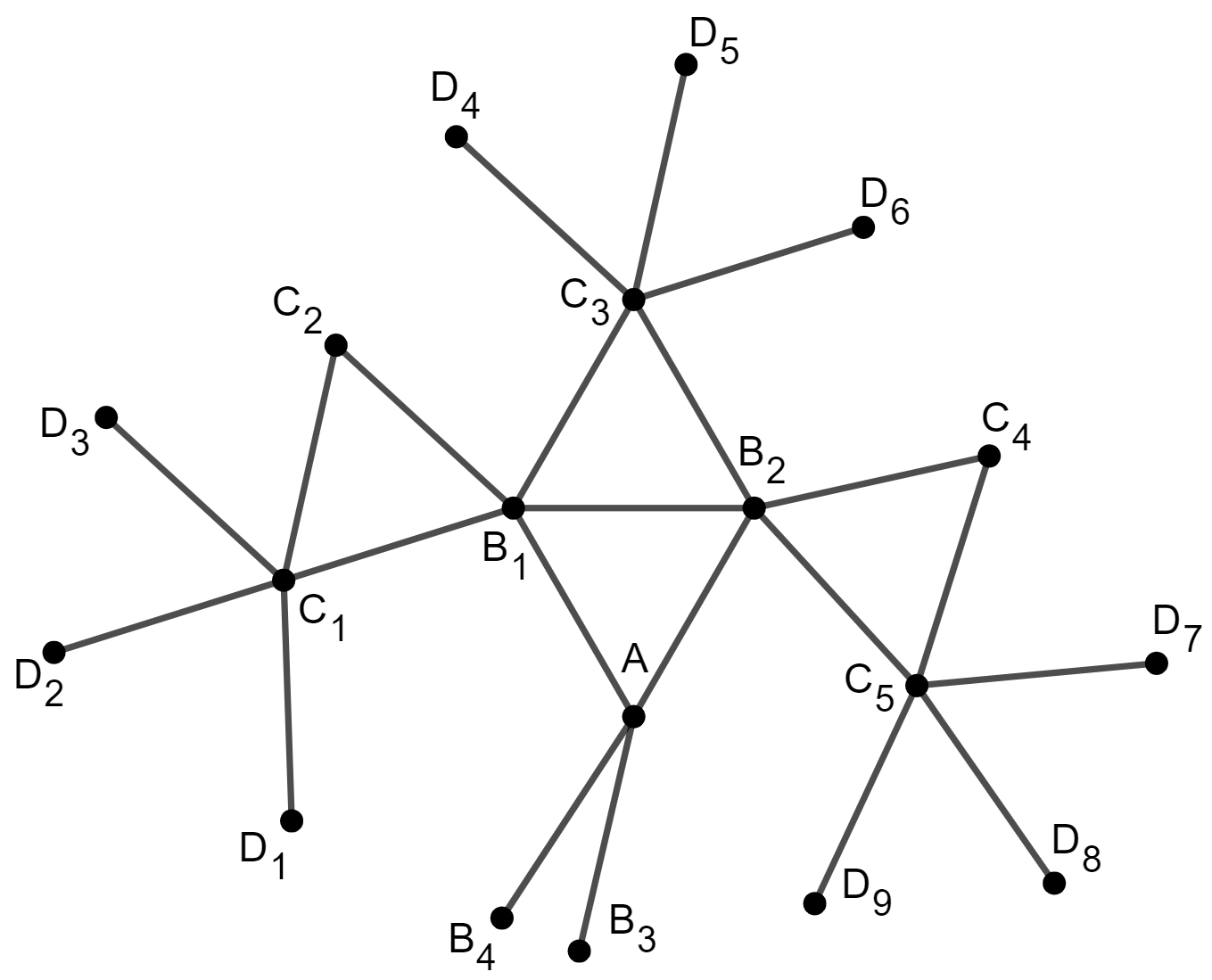}
		\caption{Before rotations}
		\label{F121}
	\end{subfigure}
	\begin{subfigure}[b]{.45\linewidth}
		\centering
		\includegraphics[scale=1.40]{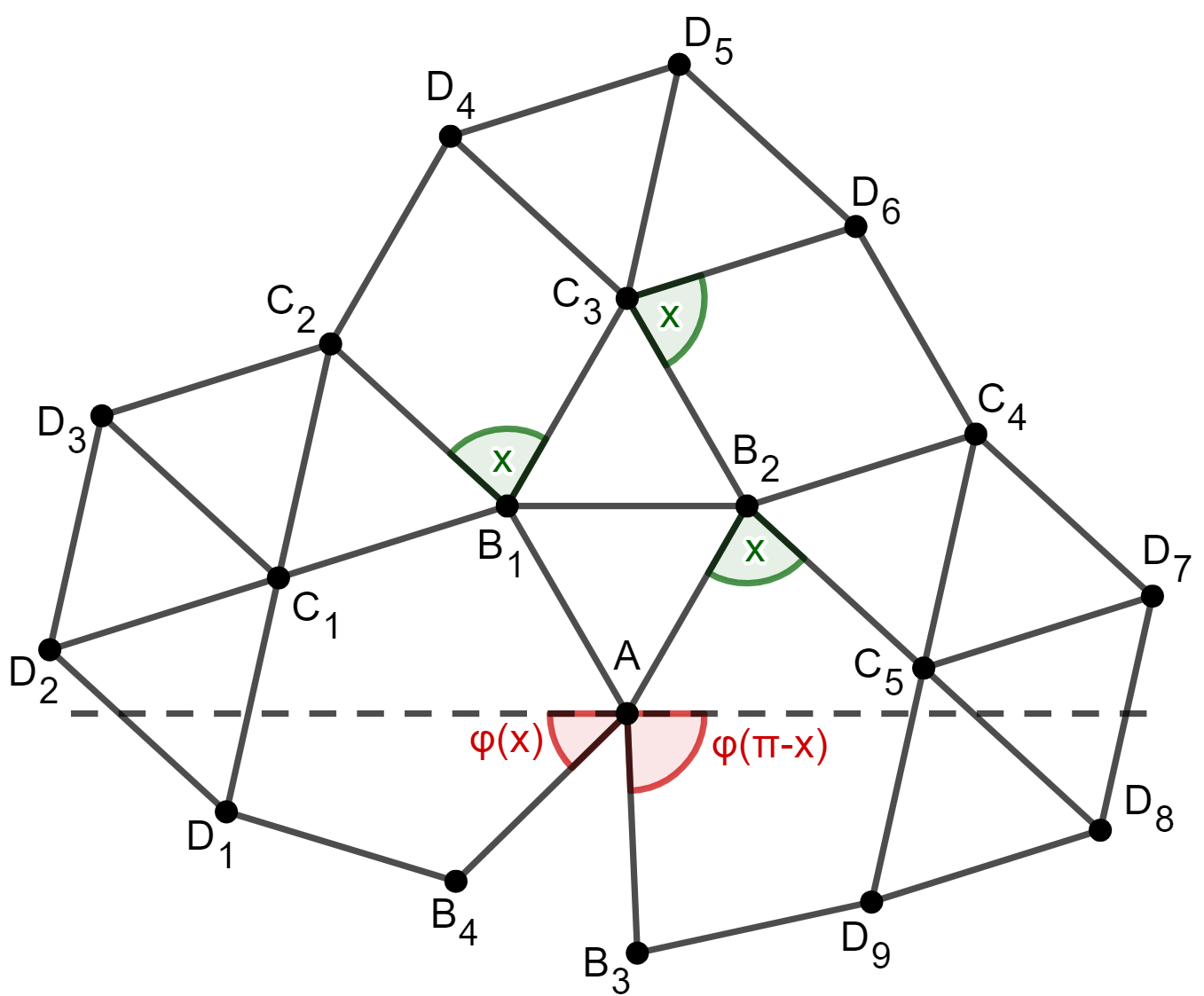}
		\caption{After rotations}
		\label{F122}
	\end{subfigure}
	\captionsetup{justification=centering}
	\caption{Forbidden 19-point configuration.}
	\label{F12}
\end{figure}

\vspace{-3mm}
\begin{proof}
	 As in the proof of \Cref{l_no_kifli}, we assume the contrary, namely that it is possible to place 19 points in general position on the plane such that all the edges in \Cref{F121} are of unit length, while all the other pairwise distances are not shorter. Then we will rotate some parts of this configuration to reduce the number of `degrees of freedom' and simplify the further analysis.
	 
	 First, for each $i=3,2,1$, we consecutively rotate $D_i$ around $C_1$ towards the next neighbor of $C_1$ in clockwise order until an edge between them appears. Note that the path $D_1C_1B_1AB_4$ is convex by \Cref{l_convex}. Therefore, during the whole rotation procedure, $D_2$ and $B_4$ lie on the opposite sides of the line $C_1D_1$, and thus the distance $|D_1B_4|$ increases.
	 
	 Second, we rotate 5 points $C_1, C_2, D_1, D_2, D_3$ around $B_1$ towards $D_4$ with the same angular speed until a new edge appears. If this edge is $C_2D_4$, then $B_1C_2D_4C_3$ is a rhombus, and we proceed to the next step. If this edge is $C_2C_3$, then we rotate $D_4$ around $C_3$ towards $C_2$ until $B_1C_2D_4C_3$ is a rhombus. As earlier, the path $D_1C_1B_1AB_4$ is convex, and thus during the whole rotation procedure, $D_2$ and $B_4$ lie on the opposite sides of the line $B_1D_1$, so the distance $|D_1B_4|$ increases.
	 
	 Third, we repeat the previous two steps for the other half of our construction, that is to the points $B_2, B_3, C_5, C_4, D_9, D_8, D_7, D_6$ instead of $B_1, B_4, C_1, C_2, D_1, D_2, D_3, D_4$, respectively.
	 
	 Fourth, we rotate $D_5$ around $C_3$ towards $D_6$ until the edge $D_5D_6$ appears. Then we rotate $D_4$ around $C_3$ towards $D_5$ and 5 points $C_1, C_2, D_1, D_2, D_3$ around $B_1$ towards $D_4$ with the same angular speed preserving all the edges between them until $D_4D_5$ also becomes an edge. As in the second step, the distance $|D_1B_4|$ increases during the whole rotation procedure.
	 
	 Fifth, we rotate $B_4$ around $A$ towards $B_1$ until a new edge appears. We claim that this edge is $D_1B_4$. Indeed, if the edge $B_1B_4$ appears the first, then the degree of $B_1$ becomes equal to $6$. Therefore, all the angles around $B_1$ are $\pi/3$, and the edge $C_1B_4$ between two cyclically consecutive neighbors of $B_1$ also appears. The same argument applied to $C_1$ yields that if $C_1B_4$ is an edge, then $D_1B_4$ is an edge too. Hence the edge $D_1B_4$ indeed appears the first, possibly simultaneously with some others. Note that this rotation increases the distance $|B_3B_4|$.
	 
	 Finally, we repeat the previous step to the other half of our construction, that is we rotate $B_3$ around $A$ towards $B_2$ until the edge $D_9B_3$ appears increasing the distance $|B_3B_4|$.
	 
	 This series of rotations transforms our initial configuration in \Cref{F121} to a much more structured one depicted in \Cref{F122} increasing the distance $|B_3B_4|$ on the way\footnote{An anonymous referee noted that the increase of $|B_3B_4|$ is also immediate from Cauchy’s Arm Lemma, see~\cite[page~228]{Crom}.}. In addition, we can assume without loss of generality that $A$ is the origin, and the line $AC_3$ is vertical. Note that in our configuration, there is only one `degree of freedom' left. More formally, the position of each point can be uniquely expressed as a function of one variable: the value $x$ of each of the three equal angles $\angle C_5B_2A = \angle D_6C_3B_2 = \angle C_2B_1C_3$ ranging between $\pi/3$ and $2\pi/3$.
	 
	 We claim that the angle $\angle B_3AB_4$ equals $\pi/3$ at the endpoints of this range and strictly smaller than $\pi/3$ at all the intermediate points. This claim implies that before the rotation procedure, the distance $|B_3B_4|$ could be no less than $1$ if and only if our initial configuration coincides with the one in \Cref{F122} and $x$ equals either $\pi/3$ or $2\pi/3$. However, this configuration is definitely not in general position, since, e.g., three points $D_1$, $C_1$, and $C_2$ are collinear, a contradiction.
	 
	 In the rest of this proof, we provide some technical details to convince most readers that our claim is correct, and to aid the others verify it via computer algebra systems. First, one can check that the horizontal and vertical coordinates of $D_1$ are equal to
	 \begin{equation*}
	 	a(x) \coloneqq -\frac{1}{2} -\sqrt{3}\sin\Big(x+\frac{\pi}{3}\Big) \ \mbox{ and } \  b(x) \coloneqq \frac{\sqrt{3}}{2} +\sqrt{3}\cos\Big(x+\frac{\pi}{3}\Big),
	 \end{equation*} 
 	respectively. Next, the angle between $AB_4$ and the negative half of the horizontal axis, see \Cref{F122}, equals
 	\begin{equation*}
 		\varphi(x) \coloneqq \arctan\Big(\frac{b(x)}{a(x)}\Big)+\arccos\Big(\frac{\sqrt{a(x)^2+b(x)^2}}{2}\Big).
 	\end{equation*}
 	Similarly, the angle between the positive half of the horizontal axis and $AB_3$ equals $\varphi(\pi-x)$, and thus
 	\begin{equation*}
 		\angle B_3AB_4 = \pi - \varphi(\pi-x) - \varphi(x).
 	\end{equation*}
 	It is straightforward to check that the latter quantity is equal to $\pi/3$ if $x$ equals either $\pi/3$ or $2\pi/3$. Besides that, the plot of this one-variable function that we drew in $\mathtt{Maple 2019}$, see \Cref{F13}, suggests\footnote{Note that this observation can be made mathematically rigorous in the following routine way that we decided to omit. First, since the derivative of the function tents to $-\infty$ as $x \to \pi/3$, a rough estimate is enough to show that the derivative is negative for $\pi/3 < x < \pi/3+\varepsilon$ for a sufficiently small positive $\varepsilon$. Pick a positive $\delta$ such that $\angle B_3AB_4 = \pi/3-\delta$ for $x=\pi/3+\varepsilon$. Similarly,  $\angle B_3AB_4 = \pi/3-\delta$ for $x=2\pi/3-\varepsilon$, and the derivative is positive for $2\pi/3-\varepsilon < x < 2\pi/3$. Hence, our claim is valid at least on these two short segments. For $\pi/3+\varepsilon \le x \le 2\pi/3-\varepsilon$, our function is smooth, and so the absolute value of its derivative is bounded by some constant $M$ which can be roughly estimated. Having this estimate, it is sufficient to verify that $\angle B_3AB_4 \le \pi/3-\delta$ for roughly $M/\delta$ values of $x$ equally distributed  between $\pi/3+\varepsilon$ and $2\pi/3-\varepsilon$ to conclude that $\angle B_3AB_4 < \pi/3$ for all the intermediate points, as claimed.} that $\angle B_3AB_4 < \pi/3$ for all $\pi/3 < x < 2\pi/3$, as desired.
\end{proof}

\begin{figure}[!htb]
	\centering
	\includegraphics[scale=1.00]{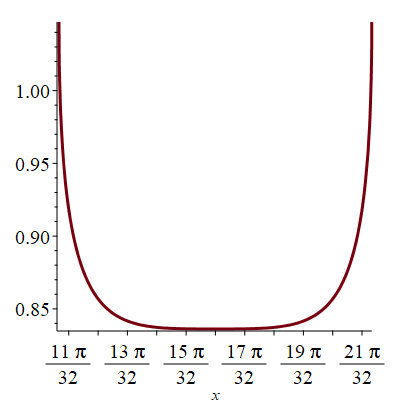}
	\caption{Dependence of $\angle B_3AB_4$ on $x$.}
	\label{F13}
\end{figure}

\begin{Proposition} \label{p_typeii}
	Four edges of $A$ cannot all be of Type~II.
\end{Proposition}

\begin{proof}
	Our argument is very close to the proof of \Cref{p_typei}. As earlier, we assume the contrary, namely that all $4$ neighbors of $A$ are popular, and all the edges of $A$ are of Type~II. Let us label these neighbors by $B_1, B_2, B_3,$ and $B_4$ in clockwise order such that $B_1B_2$ and $B_3B_4$ are both edges. If neither $B_1$ and $B_4$, nor $B_2$ and $B_3$ have a common neighbor different from $A$, then our penny graph contains a configuration depicted in \Cref{F121} as a subgraph, which contradicts \Cref{l_no_clover}. So we can assume without loss of generality that $B_2$ and $B_3$ share a common neighbor.
	
	Let us label five remaining neighbors of $B_2$ and $B_3$, including a common one, by $C_1, C_2, C_3, C_4, C_5$ in clockwise order, see \Cref{F141}. Since $B_2$ is popular, we can apply \Cref{l_mobius} to the triangle $B_2C_2C_3$, which yields that either there is a \MS{} around it, or $C_2$ and $C_3$ share a common neighbor different from $B_2$.
	
	In the former case, we label four remaining vertices of the \MS{} by $D_1, D_2, D_3, D_4$ in clockwise order, see \Cref{F142}. Observe that $\angle D_2C_2C_3 = \angle B_3C_3C_2$. Hence, the segment $D_2C_2$ is parallel to $C_3B_3$. Similarly, $C_2C_3$ is parallel to $B_3C_5$ because $\angle B_3C_3C_2 = \angle C_3B_3C_5$. Since these four segments are also of the same length, we conclude that the vertices $D_2$, $C_3$, and $C_5$ are collinear, a contradiction.
	
	Therefore, \Cref{l_mobius} implies that $C_2$ and $C_3$ have a common neighbor different from $B_2$, which we denote by $D_1$, see \Cref{F143}. Similarly, $C_3$ and $C_4$ must have a common neighbor $D_2$ different from $B_3$. It remains only to observe that whether $D_1=D_2$ or not, we get a contradiction with \Cref{l_deg5}.
\end{proof}

\begin{figure}[!htb]
	\centering
	\begin{subfigure}[b]{.32\linewidth}
		\centering
		\includegraphics[scale=1.30]{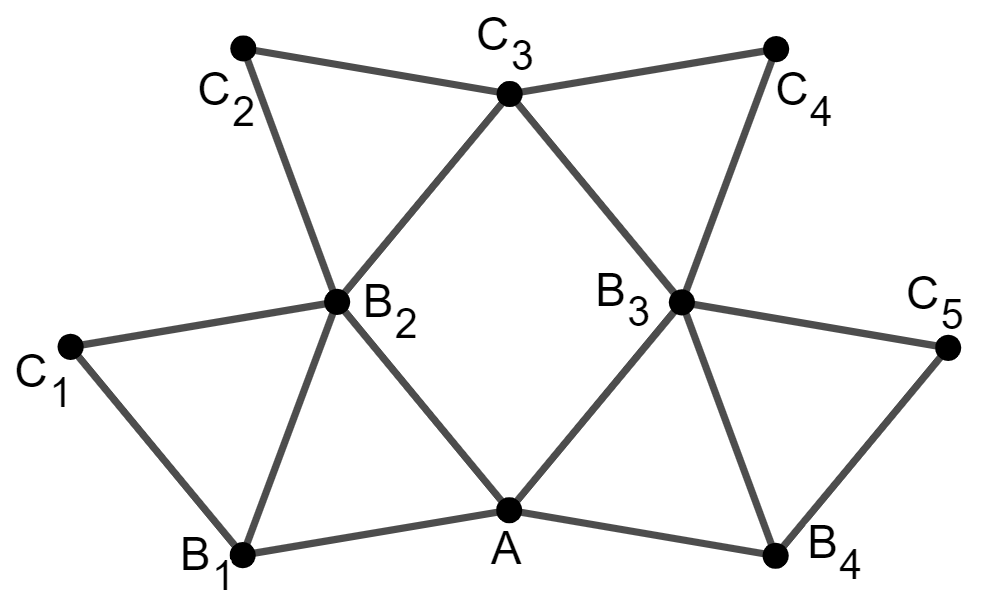}
		\captionsetup{justification=centering}
		\caption{Apply \Cref{l_mobius} to $B_2C_2C_3$ \newline and $B_3C_3C_4$}
		\label{F141}
	\end{subfigure}
	\begin{subfigure}[b]{.32\linewidth}
		\centering
		\includegraphics[scale=1.30]{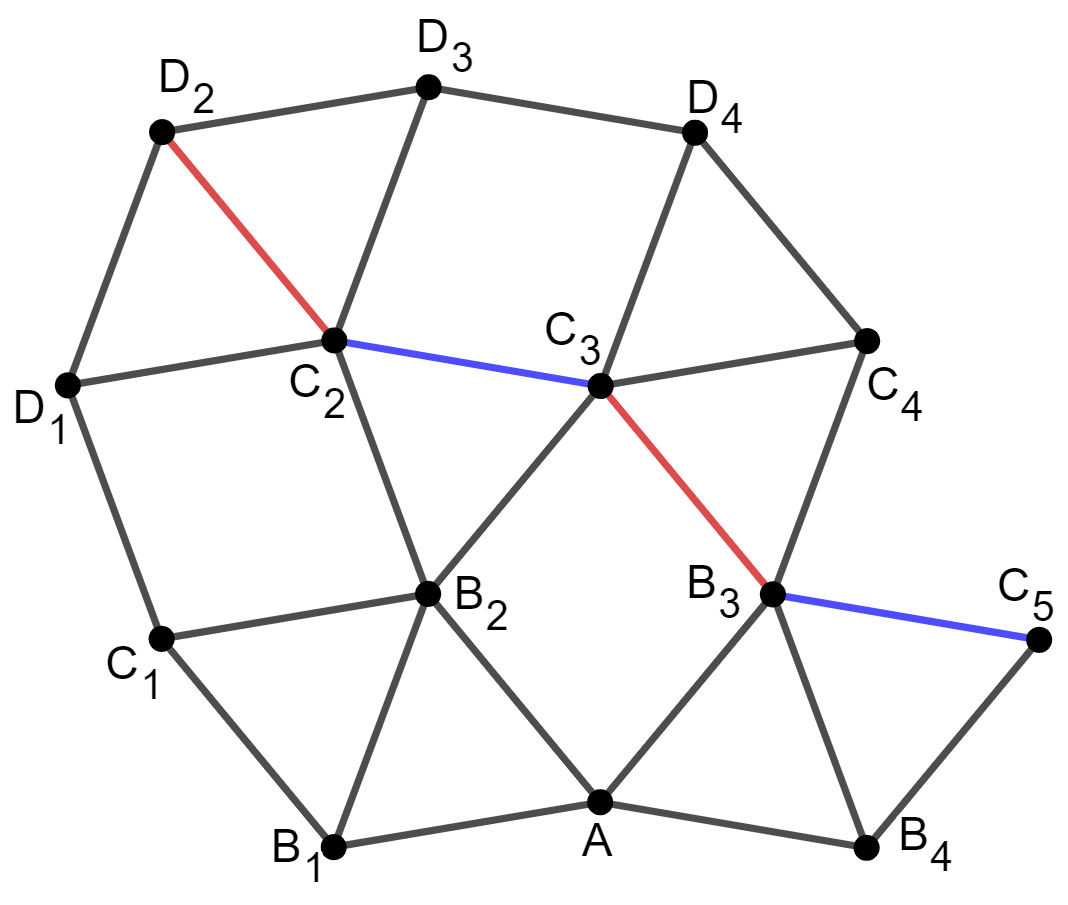}
		\captionsetup{justification=centering}
		\caption{\MS{} leads to \newline a contradiction}
		\label{F142}
	\end{subfigure}
	\begin{subfigure}[b]{.32\linewidth}
		\centering
		\includegraphics[scale=1.30]{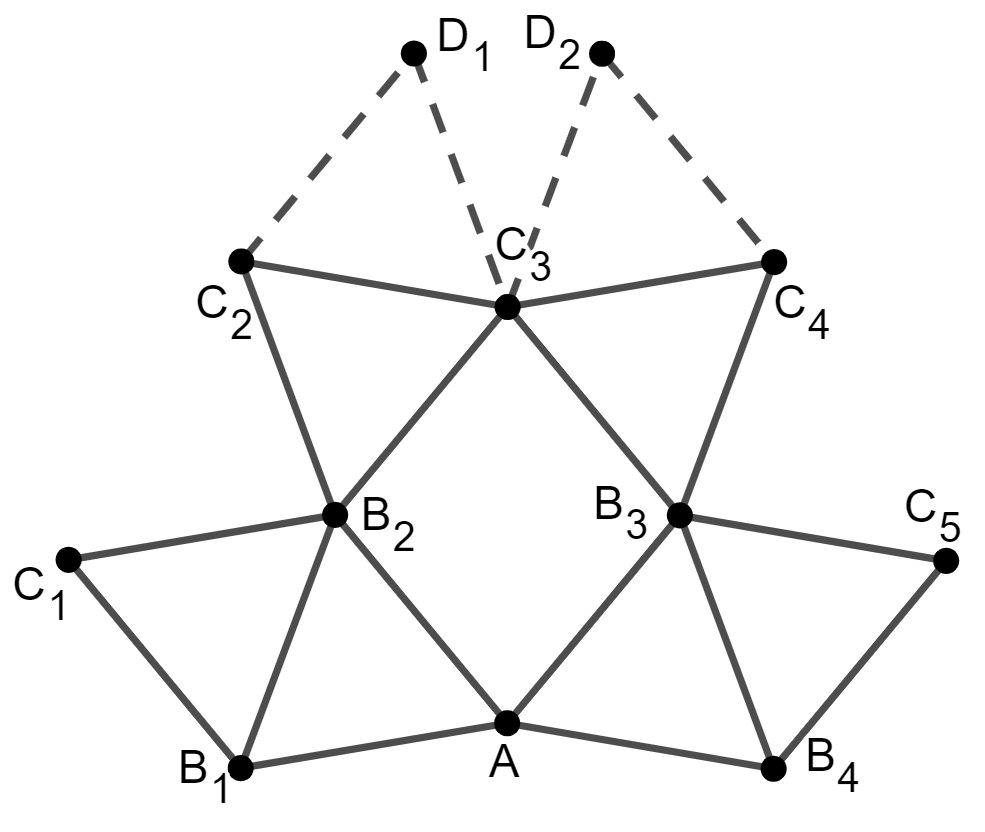}
		\captionsetup{justification=centering}
		\caption{Common neighbors yield \newline a contradiction too}
		\label{F143}
	\end{subfigure}
	\captionsetup{justification=centering}
	\caption{Two pairs of edges of Type~II joined by a rhombus is a forbidden configuration.}
	\label{F14}
\end{figure}

\section{Concluding remarks} \label{sec:con}

\noindent
\textbf{Limitations of the method}. Note that in the discharging procedure described at the beginning of \Cref{s3}, there might be no extra charge left in the second neighborhood of a vertex provided that it has only one unpopular neighbor, which can be the case. A careful analysis of the third neighborhood will probably lead to an improvement of our upper bound. However, rather different ideas may be required to finally get a tight result, which we suspect is much closer to or even coincides with the lower bound from \Cref{tT}.

\vspace{1.5mm}

\noindent
\textbf{Planar generalizations}. One can study a variation of this problem, where the strong condition that the vertices are in general position is replaced by a weaker one that no $m$ of them are collinear, $m \ge 3$. If $c_m$ stands for the maximum density of a penny graph satisfying this condition, then it is not hard to show that $c_m<3$ for all $m \in \N$, and that $c_m \to 3$ as $m$ grows. However, the exact magnitude of the latter convergence may be not that easy to grasp. Another possible direction here is to replace penny graphs with a different class of geometric graphs whose edges are all of the same length, e.g. with $k$-planar unit distance graphs, see~\cite{GT}. Since most of the extremal constructions of these graphs are based on lattices, it is natural to expect that the additional constraint that not too many vertices are collinear will make the problem non-trivial.

\vspace{1.5mm}

\noindent
\textbf{Higher dimensions}. It is easy to see that among $n$ points in $\R^d$, the shortest distance occurs at most $k_dn/2$ times, where $k_d$ is the $d$-dimensional kissing number, see~\cite{KP} and \cite[Section~4]{Swan}. This upper bound is known to be tight, up to the $o(n)$-term, only for a few small values of $d$. In these dimensions, it would be interesting to study how an additional assumption that the points are in general position affects the leading term.

\vspace{5mm}

\noindent
{\bf \large Acknowledgements.} Research was supported by ERC Advanced Grant `GeoScape' No. 882971. The author would like to thank an anonymous referee for their helpful suggestions and G\'eza T\'oth, who introduced us to the problem, for his valuable comments and constant attention to our work.

\end{document}